\documentclass{amsart}
\usepackage[margin=1.546in]{geometry}
\usepackage{amssymb}
\usepackage{amsmath}
\usepackage{graphicx}
\usepackage{manfnt}
\usepackage{amsthm}
\usepackage[mathscr]{euscript}
\usepackage{mathtools}
\usepackage{epstopdf}
\usepackage{bigints}
\usepackage{color}
\usepackage{dsfont}
\usepackage{enumerate}
\usepackage{tikz}
\usepackage{bm}
\usetikzlibrary{arrows}
\usetikzlibrary{positioning}
\usepackage{mathrsfs}
\numberwithin{equation}{section}
\usepackage{hyperref}

\newcommand{\bbm}{\begin{bmatrix}}
\newcommand{\ebm}{\end{bmatrix}}
\newcommand{\bv}{\begin{vmatrix}}
\newcommand{\ev}{\end{vmatrix}}

\newcommand{\g}{\mathfrak{g}}

\newcommand{\p}{\mathfrak{p}}

\newcommand{\n}{\mathfrak{n}}
\newcommand{\h}{\mathfrak{h}}
\newcommand{\lev}{\mathfrak{l}}

\newcommand{\z}{\mathfrak{z}}

\newcommand{\C}{\mathbb{C}}
\newcommand{\Z}{\mathbb{Z}}
\newcommand{\beq}{\begin{equation*}}
\newcommand{\eeq}{\end{equation*}}
\newcommand{\beqn}{\begin{eqnarray*}}
\newcommand{\eeqn}{\end{eqnarray*}}

\newcommand{\mf}{\mathfrak}
\newcommand{\mc}{\mathcal}
\newcommand{\bp}{\begin{pmatrix}}
\newcommand{\ep}{\end{pmatrix}}

\DeclareMathOperator{\ch}{ch}

\DeclareMathOperator{\Ann}{Ann}

\DeclareMathOperator{\Ind}{Ind}
\DeclareMathOperator{\res}{res}

\DeclareMathOperator{\rad}{rad}

\theoremstyle{plain}
\newtheorem{theorem}{Theorem}[section]
\theoremstyle{definition}
\newtheorem{definition}[theorem]{Definition}
\theoremstyle{plain}
\newtheorem{lemma}[theorem]{Lemma}
\theoremstyle{plain}
\newtheorem{proposition}[theorem]{Proposition}
\theoremstyle{remark}

\theoremstyle{plain}
\newtheorem{corollary}[theorem]{Corollary}
\theoremstyle{remark}
\newtheorem{remark}[theorem]{Remark}
\theoremstyle{definition}
\newtheorem{example}[theorem]{Example}
\theoremstyle{definition}
\newtheorem{notation}[theorem]{Notation}

\title{Contravariant forms on Whittaker modules}
\author{Adam Brown and Anna Romanov}
\date{}

\begin{document}

\maketitle

\begin{abstract}
    Let $\g$ be a complex semisimple Lie algebra. We give a classification of contravariant forms on the nondegenerate Whittaker $\g$-modules $Y(\chi, \eta)$ introduced by Kostant in \cite{Kostant}. We prove that the set of all contravariant forms on $Y(\chi, \eta)$ forms a vector space whose dimension is given by the cardinality of the Weyl group of $\mf{g}$. We also describe a procedure for parabolically inducing contravariant forms. As a corollary, we deduce the existence of the Shapovalov form on a Verma module, and provide a formula for the dimension of the space of contravariant forms on the degenerate Whittaker modules $M(\chi, \eta)$ introduced by McDowell in \cite{McDowell}.    
\end{abstract}

\section{Introduction}
\label{introduction}

This paper concerns a classical tool in the study of representations of Lie algebras: contravariant forms. Contravariant forms are certain symmetric bilinear forms on modules over a Lie algebra which are invariant under an antiautomorphism of the Lie algebra (Definition \ref{contravariant form}). Many well-studied classes of Lie algebra modules, such as Verma modules, finite-dimensional irreducible modules, and, more generally, highest weight modules in Bernstein--Gelfand--Gelfand's category $\mc{O}$, admit a unique contravariant form up to a multiplier. %Many well-studied classes of Lie algebra modules, such as Verma modules and irreducible finite-dimensional modules, admit a unique contravariant form up to scaling. 
In this paper we study a class of Lie algebra modules for which this is not the case. 

Let $\g$ be a complex semisimple Lie algebra with a fixed Borel subalgebra $\mf{b} \subset \g$ and Cartan subalgebra $\h \subset \mf{b}$. Let $\n=[\mf{b},\mf{b}]$ be the nilpotent radical of $\mf{b}$, and $W$ the Weyl group of $\g$. Denote by $U(\g)$ the universal enveloping algebra of $\g$ and by $Z(\g)$ its center. In \cite{Kostant}, Kostant introduced a family of Whittaker $\g$-modules
\begin{equation*}
    \label{intro definition}
    Y(\chi, \eta):=U(\g)\otimes_{Z(\g)\otimes_\C U(\n)} \C_{\chi, \eta},
\end{equation*}
where $\C_{\chi, \eta}$ is the one-dimensional $Z(\g)\otimes_\C U(\n)$-module determined by the characters $\chi:Z(\g)\rightarrow \C$ and $\eta:\n\rightarrow \C$ (Definition \ref{nondegenerate Whittaker module}). Each module $Y(\chi, \eta)$ is cyclically generated by a Whittaker vector $w=1 \otimes 1 \in Y(\chi, \eta)$ on which $\n$ acts by a nondegenerate Lie algebra morphism $\eta:\n\rightarrow \C$ (Definition \ref{nondegenerate}). The modules $Y(\chi, \eta)$ are infinite-dimensional and irreducible. The main result of this paper is a classification of contravariant forms on $Y(\chi, \eta)$. 
%Let $\g$ be a finite-dimensional complex semisimple Lie algebra. In \cite{Kostant}, Kostant introduced a family of nondegenerate Whittaker $\g$-modules $Y(\chi, \eta)$ (Definition \ref{nondegenerate Whittaker modules}) which are cyclically generated by a Whittaker vector $w=1 \otimes 1 \in Y(\chi, \eta)$ on which the nilpotent radical $\n$ of a fixed Borel subalgebra $\mf{b}$ acts by a nondegenerate Lie algebra morphism $\eta:\n\rightarrow \C$ (Definition \ref{nondegenerate}). The modules $Y(\chi, \eta)$ are infinite-dimensional and irreducible. The main result of this paper is a classification of contravariant forms on $Y(\chi, \eta)$. 
\begin{theorem}
\label{intro theorem}(Theorem \ref{main theorem}) The set of contravariant forms on the Whittaker module $Y(\chi, \eta)$ is a finite-dimensional vector space whose dimension is given by the cardinality of the Weyl group of $\g$. 
\end{theorem}
To prove Theorem \ref{intro theorem}, we construct a vector space isomorphism between the space of contravariant forms on $Y(\chi, \eta)$ and the space of  Weyl group coinvariants in the symmetric algebra $S(\h)$. By classical results in invariant theory, the space of $W$-coinvariants in $S(\h)$ is isomorphic to the regular representation of $W$, so this isomorphism lets us conclude that the space of contravariant forms is $|W|$-dimensional.

The reason we can construct such an isomorphism has to do with the fact that the modules $Y(\chi, \eta)$ have an infinitesimal character; that is, the center $Z(\g)$ acts on $Y(\chi, \eta)$ by $\chi:Z(\g)\rightarrow \C$. Contravariant forms on cyclic $U(\g)$-modules are closely related to linear functionals on $U(\g)$ which vanish on the annihilator of a generating vector (Proposition \ref{forms-functionals}). The annihilator in $U(\g)$ of the generating Whittaker vector $w \in Y(\chi, \eta)$ is generated by $\ker{\eta}\subset U(\n)$ and $\ker\chi\subset Z(\g)$ (Proposition \ref{annihilator}). Hence to determine the dimension of the space of contravariant forms on $Y(\chi, \eta)$, it suffices to determine the dimension of a vector space complement in $U(\g)$ to the subspace spanned by $\Ann_{U(\g)}w$ and its image under the antiautomorphism of Definition \ref{contravariant form}  (Proposition \ref{first Ug decomp}, Lemma \ref{reduction to Q}). Computing this codimension reduces to determining a complement in $U(\h)$ to the image of the ideal generated by $\ker \chi$ under an $\eta$-twisted version of the Harish-Chandra homomorphism (Definition \ref{twisted HC proj}). The bulk of our argument in Section \ref{classification of contravariant forms} is dedicated to showing that this complement can be realized as the space of $W$-coinvariants.  

As a secondary result, we establish a procedure for parabolically inducing contravariant forms from nondegenerate Whittaker modules for a reductive subalgebra $\mf{l}\subset \g$ to degenerate Whittaker modules for $\g$. 
\begin{theorem}
\label{second theorem intro} (Theorem \ref{second result}) 
Let $\eta:\n\rightarrow \C$ be an Lie algebra morphism, and $\mf{l}_\eta=\overline{\n}_\eta \oplus \h \oplus \n_\eta \subset \g$ the corresponding reductive Lie subalgebra generated by the simple root spaces on which $\eta$ does not vanish. Let $\mf{p_\eta}$ be the standard parabolic subalgebra with Levi factor $\mf{l}_\eta \subset \p_\eta$. Let $V$ be an irreducible finitely generated $U(\mf{l}_\eta)$-module with the property that for each $v \in V$, there exists $k \in \Z_{>0}$ such that the $U(\n_\eta)$ action satisfies 
\[
(x-\eta(x))^kv=0
\]
for all $x \in \n_\eta$. Then vector space $\Psi_V$ of contravariant forms on the $\mf{l}_\eta$-module $V$ is isomorphic to the vector space $\Psi_{\Ind_{\mf{l}_\eta}^\g V}$ of contravariant forms on the parabolically induced $\g$-module $\Ind_{\mf{l}_\eta}^\g V:=U(\g)\otimes_{U(\mf{p}_\eta)}V$.  
\end{theorem}
If $\eta=0$, we have $\mf{l}_\eta=\h$ and $\mf{p}_\eta=\mf{b}$. Then for $\lambda \in \h^*$, an application of Theorem \ref{second theorem intro} to the one-dimensional $U(\h)$-module $\C_\lambda$ implies that the space of contravariant forms on the Verma module 
\[
M(\lambda):=U(\g) \otimes_{U(\mf{b})} \C_\lambda 
\]
is one-dimensional. In particular, this implies the existence and uniqueness (up to scaling) of the Shapovalov form on a Verma module \cite{Shapovalov}. 

For partially degenerate $\eta$, an application of Theorem \ref{second theorem intro} to the $\mf{l}_\eta$-module $Y(\chi, \eta)$ implies that the dimension of the space of contravariant forms on the Whittaker module $M(\chi, \eta)$ (equation (\ref{McDowell})) introduced by McDowell in \cite{McDowell} is given by the cardinality of the Weyl group of $\mf{l}_\eta$. 

This paper is organized in the following way. In Section \ref{preliminaries and notation}, we establish our conventions and definitions. In Section \ref{classification of contravariant forms}, we prove Theorem \ref{intro theorem}. In Section \ref{induction of contravariant forms} we prove Theorem \ref{second theorem intro} and two corollaries. In Section \ref{example} we provide a detailed $\mf{sl}_2(\C)$ example to illustrate the main arguments of Section \ref{classification of contravariant forms} more explicitly. 

\section*{Acknowledgements}
We would like to thank Peter Trapa for useful discussions, and Dragan Mili\v{c}i\'{c} and Arun Ram for valuable feedback on the structure of the paper. The first author acknowledges the support of the European Union’s Horizon 2020 research and innovation programme under the Marie Skłodowska-Curie Grant Agreement No. 754411. The second author is supported by the National Science Foundation Award No. 1803059.

%This paper is organized in the following way. In Section \ref{preliminaries and notation}, we establish conventions and list some basic facts about standard Whittaker modules. In Section \ref{classification of contravariant forms}, we classify contravariant forms on standard Whittaker modules. In Section \ref{example}, we include a detailed $\mf{sl}(2,\C)$ example to illustrate the main ideas of the classification proof more concretely. 

\section{Preliminaries and notation}
\label{preliminaries and notation}

Let $\mf{g}$ be a complex semisimple Lie algebra. Fix a Cartan subalgebra $\mf{h}$ contained in a Borel subalgebra $\mf{b}$. Let $\Pi \subset \Sigma^+ \subset \Sigma\subset \h^*$ be the corresponding sets of simple and positive roots in the root system of $\g$. Let $W$ be the Weyl group associated to this root system, and let $\rho=\frac{1}{2}\sum_{\alpha \in \Sigma^+} \alpha$. For $\alpha \in \Sigma$, let $\g_\alpha = \{ x \in \g \mid [h,x]=\alpha(h)x\}$ be the $\alpha$-root space of $\g$, and choose a Chevalley basis $\{y_\alpha, x_\alpha\}_{\alpha \in \Sigma^+} \cup \{h_\alpha\}_{\alpha \in \Pi}$ of $\g$ with $x_\alpha \in \g_\alpha$, $y_\alpha \in \g_{-\alpha}$ and $h_\alpha \in \h$ such that $[x_\alpha, y_\alpha]=h_\alpha$. 
Let $\mf{n}=[\mf{b},\mf{b}]= \bigoplus_{\alpha \in \Sigma^+} \g_\alpha$ be the nilpotent radical of $\mf{b}$, and $\overline{\mf{n}}=\bigoplus_{\alpha \in \Sigma^+}\g_{-\alpha}$. 

For a Lie algebra $\mf{a}$, we denote by $U(\mf{a})$ the universal enveloping algebra of $\mf{a}$ with center $Z(\mf{a})\subset U(\mf{a})$. We call an algebra homomorphism $\chi:Z(\mf{a})\rightarrow \C$ an {\em infinitesimal character.}  

We use the symbol $\eta$ to refer to Lie algebra morphisms $\eta:\mf{n}\rightarrow \C$. Such a Lie algebra morphism $\eta:\n\rightarrow \C$ can be extended to an algebra homomorphism $\eta:U(\mf{n})\rightarrow \C$ which we will call by the same name. Any Lie algebra morphism $\eta:\mf{n}\rightarrow \C$ determines a subset of simple roots
\[
\Pi_\eta:=\{\alpha \in \Pi \mid \eta|_{\g_\alpha} \neq 0 \}.
\]
\begin{definition}
\label{nondegenerate}
We say that a Lie algbra morphism $\eta:\mf{n}\rightarrow \C$ is {\em nondegenerate} if $\Pi_\eta = \Pi$. 
\end{definition}

Fix a nondegenerate Lie algebra morphism $\eta:\mf{n}\rightarrow \C$ and an infinitesimal character $\chi:Z(\g)\rightarrow \C$. Let $\C_{\chi, \eta}$ be the one-dimensional $Z(\g)\otimes_\C U(\n)$-module defined by 
\[
z \otimes x \cdot v = \chi(z)\eta(x)v
\]
for $z \in Z(\g)$, $x \in U(\n)$, and $v \in \C$.
Kostant introduced the following class of $U(\g)$-modules in \cite{Kostant}. 
\begin{definition}
\label{nondegenerate Whittaker module}
Let 
\[
Y(\chi, \eta):=U(\g)\otimes_{Z(\g)\otimes_\C U(\n)} \C_{\chi, \eta},
\]
be the $U(\g)$-module given by left multiplication.  
\end{definition}
The modules $Y(\chi, \eta)$ are generated by the vector $w=1 \otimes 1$. The nilpotent radical acts on $w$ by $\eta$; that is, $$x \cdot w = \eta(x)w$$ for all $x \in \n$. 

In a $U(\g)$-module $V$, a {\em Whittaker vector} is a vector $v \in V$ with the property that for all $x \in \mf{n}$, $x \cdot v=\eta(x)v$ for some Lie algebra morphism $\eta:\mf{n}\rightarrow \C$. A $U(\g)$-module which is cyclically generated by a Whittaker vector is called a {\em Whittaker module}. Hence the modules $Y(\chi, \eta)$ are {\em nondegenerate Whittaker modules}.

Kostant showed that the modules $Y(\chi, \eta)$ are irreducible \cite[Thm. 3.6.1]{Kostant}.

\section{Classification of contravariant forms on nondegenerate Whittaker modules}
\label{classification of contravariant forms}
In this section we classify contravariant forms on the nondegenerate Whittaker modules $Y(\chi, \eta)$ introduced in Section \ref{preliminaries and notation}. We show in Theorem \ref{main theorem} that the set of contravariant forms on $Y(\chi, \eta)$ is a finite-dimensional vector space whose dimension is given by the cardinality of the Weyl group of $\mf{g}$.

\begin{definition}
\label{contravariant form}
Let $\tau:U(\g)\rightarrow U(\g)$ be the antiautomorphism defined by $\tau(x_\alpha)=y_{\alpha}$, $\tau(y_\alpha)=x_\alpha$ and $\tau(h_\alpha)=h_\alpha$. A \emph{contravariant form} on a $U(\g)$-module $X$ is a symmetric bilinear form $\langle\cdot , \cdot \rangle:X \times X \rightarrow \C$ such that 
$$
\langle uv,w\rangle =\langle v,\tau(u)w\rangle
$$
for all $u\in U(\g)$ and $v,w\in X$.
\end{definition}
Contravariant forms on $U(\g)$-modules are closely related to $\tau$-invariant linear functionals on $U(\g)$. In fact, we can reformulate the classification of contravariant forms on a cyclic $U(\g)$-module to the classification of $\tau$-invariant linear functionals on $U(\g)$ which vanish on the annihilator of a generating vector of the module.  
\begin{proposition}
\label{forms-functionals}
The vector space of contravariant forms on a cyclic $U(\g)$-module $X=U(\g)v$ is isomorphic to the vector space of linear functionals $\varphi:U(\g)\rightarrow \mathbb{C}$ satisfying the following conditions:
\begin{enumerate}[(a)]
\item $\varphi\left(\text{Ann}_{U(\g)}v\right)=0$, and 
\item $\varphi(u)=\varphi(\tau(u))$ for all $u \in U(\g)$.
\end{enumerate}
\end{proposition}
\begin{proof}
Given a contravariant form $\langle\cdot ,\cdot \rangle :X \times X \rightarrow \mathbb{C} $, define $\varphi:U(\g) \rightarrow \mathbb{C} $ by
\[
\varphi(u)=\langle uv, v\rangle.
\]
The linear functional $\varphi$ satisfies conditions (a) and (b). Conversely, given $\varphi:U(\g) \rightarrow \mathbb{C}$ satisfying (a) and (b), define a bilinear form on $X$ by 
\begin{equation}
    \label{form}
\langle x, y\rangle =\langle uv, u'v\rangle =\varphi(\tau(u')u)
\end{equation}
for $x=uv, y=u'v \in X$. It is straightforward to check that this form is symmetric, bilinear, and contravariant. However, the choices of $u,u'\in U(\g)$ such that $x=uv$ and $y=u'v$ are not always unique, so it is not immediately apparent that the form is well-defined. However, if $x=tv$ and $y=t'v$, with $t  \in U(\g)$ and $t'  \in U(\g)$, then
\begin{align*}
\langle uv, u'v\rangle - \langle tv, t'v \rangle 
&=\langle (u-t)v, u'v \rangle + \langle tv, (u'-t')v \rangle \\
&=\varphi(\tau(u')(u-t)) + \varphi (\tau(t)(u'-t')) \\
&= 0,
\end{align*}
 so equation (\ref{form}) does indeed define a contravariant form. Here the second equality follows from condition (b) and the third equality follows from (a), since $u-t$ and $u'-t'$ are in the annihilator of $v$. 
\end{proof}
\begin{notation}
\label{Phi}
For a $U(\mf{g})$-module $X$, denote by $\Psi_X$ the vector space of contravariant forms on $X$. If $X$ is a cyclic $U(\g)$-module with generating vector $v$, denote by $\Phi_X$ the vector space of $\tau$-invariant linear functionals on $U(\mf{g})$ which vanish on $\Ann_{U(\mf{g})}v$. By Proposition \ref{forms-functionals}, for a cyclic $U(\g)$-module $X$, $\Psi_X \simeq \Phi_X$. 
\end{notation}
Now we restrict our attention to the cyclic $U(\mf{g})$-module $Y(\chi, \eta)$ with generating vector $w=1 \otimes 1$ (Definition \ref{nondegenerate Whittaker module}). By Proposition \ref{forms-functionals}, to study the vector space of contravariant forms on $Y(\chi, \eta)$, we need to understand the annihilator of $w$. Kostant described this annihilator in \cite{Kostant}. Recall that any Lie algebra morphism $\eta:\mf{n}\rightarrow \C$ can be extended to an algebra homomorphism $\eta:U(\mf{n})\rightarrow \C$. More precisely, on a Poincar\'{e}--Birkhoff--Witt basis $\{x_{\alpha_1}^{k_1} \cdots x_{\alpha_n}^{k_n} \mid k_i \in \Z_{\geq 0} \}$ of $U(\mf{n})$, we define 
\[
\eta(x_{\alpha_1}^{k_1} \cdots x_{\alpha_n}^{k_n}):=\eta(x_{\alpha_1})^{k_1}\cdots \eta(x_{\alpha_n})^{k_n} \text{ and for }c \in \C, \eta(c):=c. 
\]Let $\ker{\eta}\subset U(\mf{n})$ refer to the kernel of the extended map $\eta:U(\mf{n})\rightarrow \C$.  
\begin{proposition}\cite[Thm. 3.1]{Kostant} 
\label{annihilator} Fix a nondegenerate Lie algebra morphism $\eta:\mf{n}\rightarrow \C$ and infinitesimal character $\chi:Z(\g)\rightarrow \C$. Let $w=1 \otimes 1$ be the generating Whittaker vector in $Y(\chi, \eta)$. Then 
\[
\Ann_{U(\mf{g})}w = U(\mf{g})\ker{\eta} + U(\mf{g}) \ker{\chi}.
\]
\end{proposition}
%Here $U(\g)\ker \eta$ (resp. $U(\g) \ker \chi$) denotes the $U(\g)$-span of the subspace $\ker \eta \subset U(\g)$ (resp. $\ker \chi \subset U(\g)$) in $U(\g)$. This is an ideal in $U(\g)$. 

Linear functionals in $\Phi_{Y(\chi, \eta)}$ (Notation \ref{Phi}) must vanish on $\Ann_{U(\mf{g})}w+\tau(\Ann_{U(\mf{g})}w)$. Accordingly, to determine the dimension of $\Phi_{Y(\chi, \eta)}$, it will be helpful to determine a vector space complement to $\Ann_{U(\mf{g})}w + \tau(\Ann_{U(\mf{g})}w)$ in $U(\mf{g})$. The following proposition is a first step. 

\begin{proposition}
\label{first Ug decomp} Let $\eta:\mf{n}\rightarrow \C$ be a Lie algebra morphism. There is a direct sum decomposition
\begin{equation}
    \label{twisted HC hom}
    U(\g)=U(\h)\oplus (\tau(U(\mf{g})\ker{\eta})+U(\g)\ker{\eta}).
\end{equation}
\end{proposition}
\begin{proof}
Choose an order on the set of roots so that 
\[
\{y^{\overline{I}}h^Jx^K:=y_{\alpha_n}^{i_n} \cdots y_{\alpha_1}^{i_1} h_{\alpha_1}^{j_1} \cdots h_{\alpha_r}^{j_r} x_{\alpha_1}^{k_1}\cdots x_{\alpha_n}^{k_n} \mid i_l, j_l, k_l \in \Z_{\geq 0} \}
\]
forms a Poincar\'{e}--Birkhoff--Witt basis of $U(\mf{g})$. Here $I=(i_1, \ldots, i_n)$, $J=(j_1, \ldots, j_r)$ and $K=(k_1, \ldots, k_n)$ are multi-indices, $\overline{I}=(i_n, \ldots, i_1)$, and $y=(y_{\alpha_n}, \ldots, y_{\alpha_1})$, $h=(h_{\alpha_1}, \ldots, h_{\alpha_r})$,  $x=(x_{\alpha_1}, \ldots, x_{\alpha_n})$. Then we can write $y^{\overline{I}}h^Jx^K$ in the following way:
\begin{align}
y^{\overline{I}}h^Jx^K &= y^{\overline{I}}h^J\left(x^K-\eta(x^K)\right) + \eta(x^K)y^{\overline{I}}h^J \nonumber \\
&= y^{\overline{I}}h^J\left(x^K-\eta(x^K)\right) +\eta(x^K)\tau(h^J x^I) \nonumber\\
&= y^{\overline{I}}h^J\left(x^K-\eta(x^K)\right) +\eta(x^K)\tau\left(h^J \left(x^I - \eta(x^I)\right)+\eta(x^I)h^J\right) \nonumber\\
&=y^{\overline{I}}h^J\left(x^K-\eta(x^K)\right) +\eta(x^K)\tau\left(h^J \left(x^I - \eta(x^I)\right)\right)+\eta(x^K)\eta(x^I)h^J \nonumber\\
\label{box formula}
&= \boxed{y^{\overline{I}}h^J\left(x^K-\eta(x^K)\right) +\tau\left(\eta(x^K)h^J \left(x^I - \eta(x^I)\right)\right)}+ \boxed{\eta(x^K)\eta(x^I)h^J.} 
\end{align}
The first box is in $U(\mf{g})\ker{\eta} + \tau(U(\mf{g})\ker{\eta})$ and the second box is in $U(\mf{h})$. By extending linearly, we can write any vector of $U(\g)$ as a sum of a vector in $U(\mf{g})\ker{\eta} + \tau(U(\mf{g})\ker{\eta})$ and a vector in $U(\mf{h})$. The intersection $U(\mf{h}) \cap (U(\mf{g})\ker{\eta} + \tau(U(\mf{g})\ker{\eta}))=0$, so the sum is direct.
\end{proof}
\begin{definition}
\label{twisted HC proj}
Let 
\[
p_\eta:U(\g)=U(\h)\oplus (\tau(U(\mf{g})\ker{\eta})+U(\g)\ker{\eta})\rightarrow U(\h)
\]
be projection onto the first coordinate. We refer to $p_\eta$ as the {\em $\eta$-twisted Harish-Chandra projection}. 
\end{definition}
\begin{remark}
\label{HC hom} 
 Let 
\begin{equation}
    \label{commutant definition}
    U(\g)_0=\{x \in U(\g) \mid [h, x]=0 \text{ for all } h \in \h\}.
\end{equation}
Note that if $\eta=0$, the restriction of the projection $p_0$ in Definition \ref{twisted HC proj} to $U(\g)_0$ is exactly the Harish-Chandra homomorphism \cite[Ch. VIII, \S6.4]{Bourbaki7-9}. This justifies our choice of terminology in Definition \ref{twisted HC proj}.
\end{remark}
\begin{example}
Let $\mf{g}=\mf{sl}_2(\C)$, and let 
\[
y=\bp 0 & 0 \\ 1 & 0 \ep, h=\bp 1 & 0 \\ 0 & -1 \ep, x= \bp 0 & 1 \\ 0 & 0 \ep
\]
be the standard basis. The universal enveloping algebra $U(\g)$ has a basis consisting of monomials $\{y^ih^jx^k \mid i,j,k \in \Z_{\geq 0} \}$. Let $\eta:\mf{n}\rightarrow \C$ be the Lie algebra morphism sending $x \mapsto 1$. We can express the vector $y \in U(\g)$ as 
\[
y=\tau(x)=\tau(x-1)+1=\tau(x-\eta(x))+\eta(x), 
\]
so $p_\eta(y)=\eta(x)$. 
\end{example}
For the remainder of this section, we fix a nondegenerate Lie algebra morphism $\eta:\mf{n}\rightarrow \C$. Recall that our goal is to determine the dimension of the space $\Phi_{Y(\chi, \eta)}$ (Notation \ref{Phi}). Because $\tau(U(\g)\ker{\eta}) + U(\g)\ker\eta \subset \tau(\Ann_{U(\g)}w) + \Ann_{U(\g)}w$, any linear functional $\varphi \in \Phi_{Y(\chi, \eta)}$ must vanish on $\tau(U(\g)\ker{\eta}) + U(\g)\ker\eta$. Hence for any $\varphi \in \Phi_{Y(\chi, \eta)}$ and $u \in U(\g)$, 
\begin{equation}
    \label{vanish on N}
\varphi(u)=\varphi(p_\eta(u)).
\end{equation}
Moreover, by Kostant's description of the annihilator in Proposition \ref{annihilator}, $\varphi$ must also vanish on $U(\g)\ker\chi$, so for $v \in U(\g)\ker\chi$, 
\begin{equation}
    \label{vanish on J}
    \varphi(v)=\varphi(p_\eta(v))=0 
\end{equation}
 as well. Conversely, the following lemma shows that any linear functional in $U(\g)^*$ satisfying (\ref{vanish on N}) and (\ref{vanish on J}) is in $\Phi_{Y(\chi, \eta)}$. 
\begin{lemma}
\label{reduction to Q}
Let $Q=U(\h)/p_\eta(U(\g)\ker\chi)$. As vector spaces,
$$
\Phi_{Y(\chi, \eta)} \cong Q^\ast.
$$
\end{lemma}
\begin{proof}
First we note that $Q^\ast$ is canonically isomorphic to the space of linear functionals on $U(\h)$ which vanish on $p_\eta(U(\g)\ker\chi)$. We will show that restriction of linear functionals from $U(\g)$ to $U(\h)$ defines an isomorphism of $\Phi_{Y(\chi, \eta)}$ with $Q^\ast$:
\begin{align*}
\text{res}_{U(\h)}:\Phi_{Y(\chi, \eta)} &\xrightarrow{\sim} Q^\ast\\
\varphi &\mapsto \varphi\vert_{U(\h)}.
\end{align*}
For any $\varphi \in \Phi_{Y(\chi, \eta)}$, $\varphi|_{U(\h)}$ vanishes on $p_\eta(U(\g)\ker\chi)$ by (\ref{vanish on J}), so $\text{res}_{U(\h)}$ is well-defined. 
The inverse of the restriction map is given by
\begin{eqnarray*}
Q^\ast&\rightarrow& \Phi_{Y(\chi, \eta)}\\
\phi&\mapsto & \phi\circ p_\eta.
\end{eqnarray*}
To see that this inverse map is well-defined, we must show that $\phi\circ p_\eta$ vanishes on the annihilator $\Ann_{U(\g)}w$ and is $\tau$-invariant. We can write $a\in \Ann_{U(\g)}w$ as $a=n+u$ with $n\in U(\g)\ker\eta$ and $u\in U(\g)\ker\chi$ by Proposition \ref{annihilator}, and 
$$
\phi\circ p_\eta (a) = \phi(p_\eta(n))+\phi(p_\eta(u)) = 0
$$
because $\phi$ is assumed to vanish on $p_\eta(U(\g)\ker\chi)$ and $p_\eta(n)=0$. To see that $\phi \circ p_\eta$ is $\tau$-invariant, we write any $u \in U(\g)$ as $u=h+m$, where $h\in U(\h)$ and $m\in \tau(U(\g)\ker\eta)+U(\g)\ker\eta$ using (\ref{twisted HC hom}). Then, 
$$
\phi\circ p_\eta(u)= \phi\circ p_\eta(h+m)= \phi\circ p_\eta(h)= \phi\circ p_\eta(\tau(h))=\phi\circ p_\eta(\tau(m)+\tau(h))=\phi\circ p_\eta(\tau(u))
$$
since $p_\eta(\tau(m))=0$. 

The function $\varphi \mapsto \varphi|_{U(\h)}$ for $\varphi \in \Phi_{Y(\chi, \eta)}$ and the function $\phi \mapsto \phi \circ p_\eta$ for $\phi \in Q^\ast$ are inverse functions by (\ref{vanish on N}).  
\end{proof}
Lemma \ref{reduction to Q} reduces the study of $\Phi_{Y(\chi, \eta)}$ to the study of the space of linear functionals on $U(\h)$ which vanish on $p_\eta(U(\g)\ker\chi)$. To determine the dimension of this space, we will identify a vector space complement to $p_\eta(U(\g)\ker\chi)$ in $U(\h)$ which is isomorphic to the regular representation of $W$, and hence must be $|W|$-dimensional. Before making this identification, we need to establish two technical lemmas. 

% Let $U(\g)_0$ be as in (\ref{commutant definition}). The following lemma will be needed in the proof of Lemma \ref{key lemma}. 
% \begin{lemma}
% \label{commutant}
% $
% U(\g)_0 \cap (\tau(U(\mf{g})\ker{\eta})+U(\g)\ker{\eta}) \subset U(\g)\ker{\eta}.
% $
% \end{lemma}
% \begin{proof}
%  We return to the multi-index notataion from the proof of Proposition \ref{first Ug decomp}. The vector space $U(\g)_0$ is spanned by monomials of the form
% \[
% y^{\overline{I}}h^Jx^I:=y_{\alpha_n}^{i_n}\cdots y_{\alpha_1}^{i_1}h_{\alpha_1}^{j_1}\cdots h_{\alpha_r}^{j_r}x_{\alpha_1}^{i_1}\cdots x_{\alpha_n}^{i_n}.
% \]
% By (\ref{box formula}), for such a monomial, 
% \[
% p_\eta(y^{\overline{I}}h^Jx^I)=\eta(x^I)\eta(x^I)h^J.
% \]
% Hence a monomial $y^{\overline{I}}h^Jx^I \in U(\g)_0$ is in $(\tau(U(\mf{g})\ker{\eta})+U(\g)\ker{\eta})$ exactly when $x^I \in \ker\eta$. We conclude that if $y^{\overline{I}}h^Jx^I \in U(\g)_0 \cap (\tau(U(\mf{g})\ker{\eta})+U(\g)\ker{\eta})$, then $y^{\overline{I}}h^Jx^I \in U(\g) \ker\eta$. 
% \end{proof}
% In particular, Lemma \ref{commutant} implies that 
% \begin{equation}
% \label{center in twisted sum}
% Z(\g) \subset U(\g)_0 \subset U(\h) \oplus U(\g)\ker{\eta}.
% \end{equation}
% By (\ref{center in twisted sum}), any element $z \in Z(\g)$ can be expressed as a sum $z=p_\eta(z)+n$ for $n \in U(\g)\ker{\eta}$, and for any $h \in U(\h)$, 
% \begin{equation}
% \label{h moves inside}
% p_\eta(hz) = p_\eta(hp_\eta(z)+hn) = hp_\eta(z). 
% \end{equation} 

In what follows, we identify $U(\h)$ with $S(\h)$, and consider it as a representation of $W$ in the natural way. Let 
\begin{equation}
    \label{S}
    S=\langle S(\h)^{W}_+\rangle
\end{equation}
be the ideal in $S(\h)$ generated by the $W$-invariant homogeneous polynomials with positive degree. Clearly $S$ is $W$-stable, and by \cite[Ch.V, \S5.2, Thm. 2(i)]{Bourbaki4-6}, it admits a $W$-stable complement $C \subset S(\h)$ such that  
\begin{equation}
\label{C plus S}
S(\h)=C\oplus S 
\end{equation}
is a graded, $W$-stable decomposition.  Moreover, for any such $W$-stable complement $C \subset S(\h)$, the representation of $W$ in $C$ is isomorphic to the regular representation $\mathbb{C}[W]$ of $W$ \cite[Ch.V, \S5.2, Theorem 2(ii)]{Bourbaki4-6}. In particular, $\dim{C}=|W|$. Fix such a $W$-stable complement $C$. We refer to $C$ as the space of {\em $W$-coinvariants}.
\begin{remark}
There are many possible $W$-stable complements to $S$ in $S(\h)$. For example, one such complement is the space of $W$-harmonic polynomials\footnote{Under a suitable choice of orthonormal basis of $\mf{h}$, these $W$-harmonic polynomials are solutions to Laplace's equation, so they are harmonic in the usual sense of the word.} which arise as solutions to a certain system of partial differential equations determined by a generating set of $S(\h)^W$. For more details on this perspective, see \cite[\S 8.2]{Bergeron}. For our purposes, it does not matter which complement we choose, as we are only interested in its dimension.   
\end{remark}

The following technical lemma will be needed for induction arguments in the proof of Lemma \ref{complement decomposition}.
 %The following technical lemma is the crux of our argument. 
\begin{lemma}
\label{key lemma} Let $S$ be as in (\ref{S}), and $p_\eta$ the $\eta$-twisted Harish-Chandra projection (Definition \ref{twisted HC proj}).
\begin{enumerate}[(a)]
\item If $s\in S$, then there exists an element $r\in p_\eta(U(\g)\ker\chi)$, and an element $e\in U(\h)$ such that
$$
s=r+e \text{ and deg$(e)<\text{deg}(s)$.}
$$
\item Additionally, if $r'\in p_\eta(U(\g)\ker\chi)$, then there exists an element $s'\in S$, and an element $f\in U(\h)$ such that
$$
r'=s'+f \text{ and deg$(f)<\text{deg}(r')$.}
$$
\end{enumerate}
\end{lemma}
\begin{proof}
Let $t_\rho$ be the algebra automorphism of $U(\h)$ induced by the $\rho$-twisting map $h\mapsto h-\rho(h)$ for $h\in\h$. The composition of the Harish-Chandra homomorphism $p_0$ (Definition 3.6, Remark 3.7) with $t_\rho$ provides an algebra isomorphism 
\[
t_\rho \circ p_0: Z(\g) \xrightarrow{\sim} S(\h)^W
\]
\cite[Ch. VIII, \S8.5, Thm. 2]{Bourbaki7-9}.% Denote by $Z(\g)_+$ the preimage of $S(\h)_+^W$ under this isomorphism.

The ideal $S$ is generated by $S(\h)^W_+$, so any element of $S$ can be expressed as a sum of elements of the form
$$
h t_\rho(p_0(z))
$$
for various $z\in Z(\g)$ and $h\in S(\h)=U(\h)$. Our first step in the proof is to show that any element of the form $h t_\rho(p_0(z))$ satisfies
(a). 

Any $u \in U(\g)$ is a linear combination of Poincar\'{e}--Birkhoff--Witt basis elements of the form 
\[
y^{\overline{I}}h^Jx^K:=y_{\alpha_n}^{i_n}\cdots y_{\alpha_1}^{i_1}h_{\alpha_1}^{j_1}\cdots h_{\alpha_r}^{j_r}x_{\alpha_1}^{k_1}\cdots x_{\alpha_n}^{k_n}.
\]
Hence we can express $u\in U(\g)$ as a sum 
\begin{equation}
    \label{z proj}\tag{3.11}
    u=p_0(u)+\sum_{I,J,K} a_{I,J,K}(u) y^{\overline{I}}h^Jx^K
\end{equation}
where $a_{I,J,K}(u) \in \C$ and $a_{I,J,K}(u)=0$ if $I=J=(0,\ldots, 0)$. By applying $p_\eta$ to (\ref{z proj}) and using equation (3.3), we obtain 
\begin{equation}
\label{eta to zero}\tag{3.12}
p_\eta(u)=p_0(u)+\sum_{I,J,K} a_{I,J,K}(u)\eta(x^{K}x^I)h^J.
\end{equation}

Because the composition $t_\rho \circ p_0:Z(\g) \rightarrow S(\h)^W$ induces an isomorphism between the corresponding graded objects (with the grading of $Z(\g)$ induced by the natural filtration of $U(\g)$ by $\g$) \cite[Ch.VIII \S8.5 proof of Thm. 2]{Bourbaki7-9}, for $z\in Z(\g)$, we have
\begin{equation}
\label{lower degree h}\tag{3.13}
\text{deg}(h^J)<\text{deg}(p_0(z))\text{ for all $J$ such that $a_{I,J,K}(z)\neq 0$}.
\end{equation}

Hence the image of $z \in Z(\g)$ under the $\eta$-twisted Harish-Chandra projection $p_\eta$ and the image of $z$ under the Harish-Chandra homomorphism $p_0$ agree up to lower degree terms. To increase readability in the arguments below, we will introduce some notation to describe this phenomenon in general. Write $LDP$ for an element in $U(\h)$ with degree strictly lower than the element immediately preceding it in an expression\footnote{Here LDP stands for ``lower degree polynomial.''}. For example, by (\ref{lower degree h}), we can rewrite (\ref{eta to zero}) as 
\begin{equation*}
\label{eta to zero ldp}\tag{3.14}
p_\eta(z)=p_0(z)+LDP.
\end{equation*}
Similarly, for all $h\in U(\h)$, $t_\rho(h)=h+\text{LDP}$ and $p_\eta(hz)=hp_0(z)+\text{LDP}$. Therefore, we have
\begin{equation}
\label{first reduction}\tag{3.15}
ht_\rho(p_0(z))=hp_0(z)+LDP = p_\eta(hz)+LDP. 
\end{equation}
By the linearity of $p_\eta$, we have 
\begin{equation}
\label{eta-zero-relation}\tag{3.16}
    ht_\rho(p_0(z))=p_\eta(h(z-\chi(z))) + LDP.
\end{equation}
We conclude that any element of $S$ which is equal to $ht_\rho(p_0(z))$ for some $h\in U(\h)$ and $z\in Z(\g)$ satisfies (a). 

An arbitrary element $s\in S$ is a sum of elements of the form $ht_\rho(p_0(z))$ for various $h\in U(\h)$ and $z\in Z(\g)$, so by the linearity of $p_\eta$, there exists $k\in U(\g)\ker\chi$ such that 
\[
s = p_\eta(k)+LDP.
\]
This proves (a). 

Part (b) follows from an analogous argument. Any $r' \in p_\eta(U(\g)\ker\chi)$ is equal to a sum of elements of the form 
\[
p_\eta(u(z-\chi(z)))
\]
for various $u \in U(\g)$ and $z \in Z(\g)$. We claim that $p_\eta(u(z-\chi(z)))=p_\eta(p_\eta(u)(z-\chi(z)))$. Indeed, using (\ref{twisted HC hom}) to express $u$ as a sum $u=p_\eta(u)+m+n$ for $m\in \tau(U(\g)\ker{\eta})$ and $n \in U(\g)\ker{\eta}$, we have
\begin{align*}
p_\eta(u(z-\chi(z)))&=p_\eta((p_\eta(u)+m+n)(z-\chi(z)))\\
&=p_\eta(p_\eta(u)(z-\chi(z))+m(z-\chi(z))+(z-\chi(z))n)\\
&=p_\eta(p_\eta(u)(z-\chi(z))).
\end{align*}
By (\ref{eta-zero-relation}), we conclude that 
\[
p_\eta(u(z-\chi(z)))=p_\eta(u)t_\rho(p_0(z)) + LDP.
\]
For each $z\in Z(\g)$, $t_\rho(p_0(z))\in S(\h)^W$, and therefore is equal to the sum of an element of $ S$ and a constant polynomial. Hence for each $z\in Z(\g)$ and $u\in U(\g)$, there exists $s\in S$ such that
\[
p_\eta(u(z-\chi(z)))=s + LDP.
\]
Because each $r'\in p_\eta(U(\g)\ker\chi)$ is equal to a sum of elements of the form $p_\eta(u(z-\chi(z)))$, there exists $s'\in S$ such that $r'=s'+LDP$, which proves (b). 
\end{proof}
Our final step in establishing the dimension of the space $\Phi_{Y(\chi, \eta)}$ is to show that the $|W|$-dimensional vector space $C$, which is defined by the decomposition (\ref{C plus S}) to be a vector space complement to $S$ in $U(\h)$, also forms a vector space complement to $p_\eta(U(\g)\ker\chi)$ in $U(\h)$. Because we can realize linear functionals in $\Phi_{Y(\chi, \eta)}$ as linear functionals on $U(\h)$ which vanish on $p_\eta(U(\g)\ker\chi)$ by Lemma \ref{reduction to Q}, the dimension of a complement of $p_\eta(U(\g)\ker\chi)$ in $U(\h)$ determines the dimension of $\Phi_{Y(\chi, \eta)}$. 

\begin{lemma}
\label{complement decomposition} Let $C$ be the $W$-stable complement to $S=\langle S(\h) ^W_+\rangle$ in equation (\ref{C plus S}). As vector spaces, 
$$
U(\h)=C\oplus p_\eta(U(\g)\ker\chi).
$$
\end{lemma}
\begin{proof}
We begin with the graded decomposition 
$$
U(\h)=C\oplus S
$$
and proceed by induction on degree. The base case is trivial, as $p_\eta(U(\g)\ker\chi)$ contains no nonzero constant polynomials. Let $U(\h)_i$ denote the set of polynomials with degree less than or equal to $i$. Assume $U(\h)_j=C_j\oplus p_\eta(U(\g)\ker\chi)_j$ for all $j\leq i$. Let $h\in U(\h)_{i+1}$. Then 
$$
h=c+s
$$
where $c\in C_{i+1}$ and $s\in S_{i+1}$. By Lemma \ref{key lemma}, we can write $s$ as 
$$
s=r+e
$$
with $r\in p_\eta(U(\g)\ker\chi)$ and $\text{deg}(e)<\text{deg}(s)\le i+1$. Therefore 
$$
h=c+r+e.
$$
By the induction assumption, $e$ can be written uniquely as $e=c'+r'$, with $c'\in C_j$ and $r'\in p_\eta(U(\g)\ker\chi)_j$ for some $j\leq i$. So we have a decomposition of $h$ given by
$$
h=(c+c')+(r+r').
$$
Moreover, since $\text{deg}(c'),\text{deg}(r')<i+1$, we have that $(c+c')\in C_{i+1}$ and $(r+r')\in p_\eta(U(\g)\ker\chi)_{i+1}$. 

To complete the proof, it remains to show that $C_{i+1}\cap p_\eta(U(\g)\ker\chi)_{i+1}=0$. Assume $x\in C_{i+1}\cap p_\eta(U(\g)\ker\chi)_{i+1}$. By Lemma \ref{key lemma}, $x=s+f$, where $f\in U(\h)$ has lower degree than $x$ and $s\in S$. We can decompose $f$ as $f=c+s'$ with $c\in C_i$ and $s\in S_i$. So $x=s+c+s'$. Therefore $s+s'\in S\cap C$, which implies that $s+s'=0$ and $x=c$. But $\deg{c} \le i$, so $x\in C_i\cap p_\eta(U(\g)\ker\chi)_i$. Hence $x$ must be $0$ by the induction hypothesis. 
\end{proof}

Lemma \ref{complement decomposition} implies our main result. 
\begin{theorem}
\label{main theorem}
Let $\Psi_{Y(\chi, \eta)}$ be the vector space of contravariant forms on the nondegenerate Whittaker module $Y(\chi, \eta)$. Then
$$
\dim\Psi_{Y(\chi, \eta)}=|W|.
$$
\end{theorem}
\begin{proof}
The vector space $C$ has dimension $\dim{C}=|W|$ because $C$ is isomorphic to $\mathbb{C}[W]$ \cite[Ch. V, \S5.2, Theorem 2(ii)]{Bourbaki4-6}. Both $Q^\ast$ and $C^\ast$ are isomorphic to the space of linear functionals on $U(\h)$ which vanish on $p_\eta(U(\g) \ker{\chi})$. Therefore $Q^\ast \cong C^\ast$. By Lemma \ref{reduction to Q}, and Proposition \ref{forms-functionals} we conclude that  
$$
\Psi_{Y(\chi, \eta)} \cong C^\ast.
$$
Hence $\dim{\Psi_{Y(\chi, \eta)}} = |W|$. 
\end{proof}

\section{Induction of contravariant forms}
\label{induction of contravariant forms}
In this section, we prove that the induction functor from the category of nondegenerate Whittaker modules for a Levi factor $\mf{l}\subset \mf{g}$ to degenerate Whittaker modules for $\mf{g}$ induces an isomorphism of the corresponding spaces of contravariant forms. In particular, this implies the existence of the Shapovalov form on a Verma module as a corollary to Theorem \ref{main theorem}. It also gives a formula for the dimension of the space of contravariant forms on the standard degenerate Whittaker modules introduced by McDowell in \cite{McDowell}. 

For a reductive Lie algebra $\mf{a}$ with a fixed Cartan subalgebra $\mf{c}\subset \mf{a}$ and triangular decomposition $\mf{a}=\overline{\mf{m}} \oplus \mf{c} \oplus \mf{m}$, let $\mc{N}(\mf{a})$ be the category of finitely generated $U(\mf{a})$-modules which are locally $Z(\mf{a})$-finite and locally $U(\mf{m})$-finite. This category was introduced by Mili\v{c}i\'{c}--Soergel in \cite{CompositionSeries} as a natural category containing both Kostant's nondegenerate Whittaker modules $Y(\chi, \eta)$ from Section \ref{classification of contravariant forms} and all modules in Bernstein--Gelfand--Gelfand's category $\mc{O}$ \cite{BGG}. For a fixed Lie algebra morphism $\eta:\mf{m}\rightarrow \C$, the collection of modules $V \in \mc{N}(\mf{a})$ with the property that for each $v \in V$, there exists $k \in \Z_{>0}$ such that for all $x \in \mf{m}$, $(x-\eta(x))^kv=0$,
forms a full subcategory $\mc{N}(\mf{a})_\eta\subset \mc{N}(\mf{a})$. For $\eta$ nondegenerate (Definition \ref{nondegenerate}), the category $\mc{N}(\mf{a})_\eta$ contains the $\mf{a}$-modules $Y(\chi, \eta)$ for all infinitesimal characters $\chi:Z(\mf{a})\rightarrow \C$. The categories $\mc{N}(\mf{a})_\eta$ for $\eta$ nondegenerate are very simple: each such $\mc{N}(\mf{a})_\eta$ is equivalent to the category of finite-dimensional $Z(\mf{a})$-modules\footnote{This fact is originally due to Kostant \cite{Kostant}, a proof in language more closely aligned with this paper can be found in \cite[Thm. 5.9]{TwistedSheaves}.}, and the $Y(\chi, \eta)$ exhaust the irreducible objects in $\mc{N}(\mf{a})_\eta$.  

Now let $\g=\overline{\n}\oplus \h \oplus \n$ be a semisimple Lie algebra and return to the setting of Section \ref{preliminaries and notation}. Fix a Lie algebra morphism $\eta:\mf{n}\rightarrow \C$, and as in Section \ref{preliminaries and notation}, let $\Pi_\eta$ be the set of simple roots with the property that $\eta\neq 0$ on the corresponding root subspace of $\mf{g}$. Let $\Sigma_\eta \subset \h^*$ be the root system generated by $\Pi_\eta$ and $W_\eta \subset W$ the corresponding Weyl group. The morphism $\eta$ determines several Lie subalgebras of $\g$. In particular, we name
\[
\mf{l}_\eta = \h \oplus \bigoplus_{\alpha \in \Sigma_\eta}\g_\alpha, \hspace{2mm} \mf{n}_\eta = \bigoplus_{\alpha \in \Sigma_\eta^+}\g_\alpha, \hspace{2mm} \mf{n}^\eta = \bigoplus_{\alpha \in \Sigma^+-\Sigma_\eta^+} \g_\alpha, \hspace{2mm} \mf{p}_\eta = \mf{l}_\eta \oplus \n^\eta,  
\]
and define $\overline{\mf{n}}_\eta$ and $\overline{\mf{n}}^\eta$ in the obvious way. Then $\mf{l}_\eta$ is a reductive Lie subalgebra of $\mf{g}$, and $\eta|_{\mf{n}_\eta}$ is nondegenerate. 

There is an induction functor from the category of nondegenerate Whittaker modules for the Levi factor $\mf{l}_\eta$ to the category of degenerate Whittaker modules for all of $\mf{g}$. 
\begin{definition}
\label{induction functor}
Define the {\em induction functor}
\[
\Ind_{\mf{l}_\eta}^\mf{g}: \mc{N}(\mf{l}_\eta)_\eta \rightarrow \mc{N}(\mf{g})_\eta
\]
by $\Ind_{\mf{l}_\eta}^\g(V)=U(\g)\otimes_{U(\mf{p}_\eta)}V$ for a module $V \in \mc{N}(\mf{l}_\eta)_\eta$. 
\end{definition}
When applied to irreducible modules, the functor $\Ind_{\mf{l}_\eta}^\g$ induces an isomorphism on the space of contravariant forms. 
\begin{theorem}
\label{second result}
Fix a Lie algebra morphism $\eta:\mf{n}\rightarrow \C$. For any irreducible module $V \in \mc{N}(\mf{l}_\eta)_\eta$, there is a vector space isomorphism 
\[
\Psi_{V}\simeq \Psi_{\Ind_{\mf{l}_\eta}^\g V}
\]
between the vector space $\Psi_{V}$ of contravariant forms on the $\mf{l}_\eta$-module $V$ and the vector space $\Psi_{\Ind_{\mf{l}_\eta}^\g V}$ of contravariant forms on the $\mf{g}$-module $\Ind_{\mf{l}_\eta}^\g V$.
\end{theorem}
\begin{proof}
%First we note that for any $U(\g)$-module $M$ which admits a $U(\g)$-module decomposition $M=\bigoplus_{i=1}^n M_i$, there is a vector space isomorphism
%\[
%\Psi_M \simeq \bigoplus_{i=1}^n \Psi_{M_i}.
%\]
%Hence it suffices to prove the theorem for irreducible $V \in \mc{N}(\mf{l}_\eta)_\eta$ because the category $\mc{N}(\mf{l}_\eta)_\eta$ is semisimple. 

All irreducible modules in $\mc{N}(\mf{l}_\eta)_\eta$ are of the form $Y(\chi, \eta)$ for some infinitesimal character $\chi:Z(\mf{l}_\eta)\rightarrow \C$. Fix a $\mf{l}_\eta$-module $Y(\chi, \eta)$. Because $Y(\chi, \eta)$ is cyclically generated by the vector $w=1 \otimes 1$,  Proposition \ref{forms-functionals} implies that 
\[
\Psi_{Y(\chi, \eta)}\simeq \Phi_{Y(\chi, \eta)}, 
\]
where $\Phi_{Y(\chi, \eta)}$ is the vector space of $\tau$-invariant linear functionals on $U(\mf{l}_\eta)$ vanishing on $\Ann_{U(\mf{l}_\eta)}w$, as in Notation \ref{Phi}. The $\mf{g}$-module $\Ind_{\mf{l}_\eta}^\g Y(\chi, \eta)$ is cyclically generated by the vector $\overline{w}=1 \otimes w$, so again by Proposition \ref{forms-functionals}, the vector space of contravariant forms on $\Ind_{\mf{l}_\eta}^\g Y(\chi, \eta)$ is isomorphic to the vector space of linear functionals on $U(\g)$ vanishing on $\Ann_{U(\g)}\overline{w}$:
\[
\Psi_{\Ind_{\mf{l}_\eta}^\g Y(\chi, \eta)}\simeq \Phi_{\Ind_{\mf{l}_\eta}^\g Y(\chi, \eta)}. 
\]
Hence the result follows from the following proposition. 
\begin{proposition}
\label{reduction to nondegenerate}
The restriction map  $\varphi\mapsto \varphi|_{U(\mf{l}_\eta)}$ induces an isomorphism  
\[
\res_{U(\mf{l}_\eta)}:\Phi_{\Ind_{\mf{l}_\eta}^\g Y(\chi, \eta)}\xrightarrow{\sim} \Phi_{Y(\chi, \eta)}.
\]
\end{proposition}
\begin{proof}
Let $\varphi\in\Phi_{\Ind_{\mf{l}_\eta}^\g Y(\chi, \eta)}$. The $U(\mf{l}_\eta)$-module $Y(\chi,\eta)$ is naturally embedded in the induced module $\Ind_{\mf{l}_\eta}^\g Y(\chi, \eta)$, and this embedding maps $w \in Y(\chi, \eta)$ to $\overline{w}\in \Ind_{\mf{l}_\eta}^\g Y(\chi, \eta)$. Hence $\text{Ann}_{U(\mf{l}_\eta)}w\subset \text{Ann}_{U(\g)}\overline{w}$. 

Because $\varphi$ vanishes on $\text{Ann}_{U(\g)}\overline{w}$, we have that $\varphi\left(\text{Ann}_{U(\mf{l}_\eta)}w\right)=0$.
Moreover, by the $\tau$-invariance of $\varphi$, we have 
\[
\varphi(u)=\varphi(\tau(u))\text{ for all }u\in U(\mf{l}_\eta)\subset U(\g).
\]
Therefore, the image of $\res_{U(\mf{l}_\eta)}$ is contained in $\Phi_{Y(\chi, \eta)}$ as claimed. We complete the proof by constructing an inverse to $\res_{U(\mf{l}_\eta)}$.

Using the Poincar\'{e}--Birkhoff--Witt theorem, we can decompose $U(\g)$ as 
\begin{equation}
    \label{Ul decomp}
U(\g)=U(\mf{l}_\eta)\oplus(\overline{\n}^\eta U(\g)+U(\g)\n^\eta).
\end{equation}
Let $\pi_\eta$ denote the corresponding projection map from $U(\g)$ to $U(\mf{l}_\eta)$. Note that
\begin{align}\label{trans-stable-proj}
    \pi_\eta(\tau(u))=\tau(\pi_\eta(u)),
\end{align}
for any $u \in U(\g)$. 

The remainder of the proof will be devoted to showing that the map 
\begin{align*}
    \mathrm{ext}_{\pi_\eta}: \Phi_{Y(\chi, \eta)} &\rightarrow \Phi_{\Ind_{\mf{l}_\eta}^\g Y(\chi, \eta)} \\
    \phi &\mapsto \phi \circ \pi_\eta.
\end{align*}
is well-defined and is the inverse of $\res_{U(\lev_\eta)}$.
%It is clear that the compositions $\mathrm{ext}_{\pi_\eta} \circ \res_{U(\mf{l}_\eta)}$ and $\res_{U(\mf{l}_\eta)} \circ \mathrm{ext}_{\pi_\eta}$ are both the identity map. It remains to show that the image of $\mathrm{ext}_{\pi_\eta}$ is contained in $\Phi_{U(\g) \otimes_{U(\mf{p}_\eta)} Y(\chi, \eta)}$. This amounts to establishing two things: for any $\phi \in \Phi_{Y(\chi, \eta)}$, (1) $\phi \circ \pi_\eta$ is $\tau$-invariant, and (2) $\phi \circ \pi_\eta$ vanishes on $\Ann_{U(\mf{g})}\overline{w}$. 
 The $\tau$-invariance of $\phi \circ \pi_\eta$ follows immediately from the $\tau$-invariance of $\phi$ and equation (\ref{trans-stable-proj}). To show that $\phi\circ\pi_\eta(\Ann_{U(\g)}\overline{w})=0$, and therefore $\mathrm{ext}_{\pi_\eta}\phi\in \Phi_{\Ind_{\mf{l}_\eta}^\g Y(\chi, \eta)}$, we will establish the following equality: 
\begin{equation}
\label{no overlap}
    U(\mf{l}_\eta) \overline{w} \cap \overline{\mf{n}}^\eta U(\g) \overline{w}= 0. 
\end{equation}

To prove equation (\ref{no overlap}), we recall some facts about the structure of the modules $\Ind_{\mf{l}_\eta}^\g Y(\chi, \eta)$ which were established in \cite{McDowell}. Let $\z = \rad \lev_\eta$, and $\omega\in \z^\ast$ be the restriction of $\chi:Z(\lev_\eta) \rightarrow \C$ to $\z\subset Z(\lev_\eta)$. By \cite[Prop. 2.4]{McDowell}, $\z$ acts semisimply on $\Ind_{\mf{l}_\eta}^\g Y(\chi, \eta)$, and there is a partial order on the $\mf{z}$-weights which index the irreducible factors in this decomposition. The unique maximal nonzero $\z$-weight space of $\Ind_{\mf{l}_\eta}^\g Y(\chi, \eta)$ with respect to this partial order has weight $\omega$, and is equal to the $U(\lev_\eta)$-span of $\overline{w}$:
\[
\left(\Ind_{\mf{l}_\eta}^\g Y(\chi, \eta)\right)_\omega = U(\lev_\eta)\overline{w}.
\] 
Because $\overline{w}$ generates $\Ind_{\mf{l}_\eta}^\g Y(\chi, \eta)$ as a $U(\g)$-module,
\[
\bar{\n}^\eta U(\g)\overline{w} = \bar{\n}^\eta\left( \Ind_{\mf{l}_\eta}^\g Y(\chi, \eta)\right).
\] 
By \cite[Prop. 1.8(c)]{McDowell},
\[
\bar{\n}^\eta\left( \Ind_{\mf{l}_\eta}^\g Y(\chi, \eta)\right) \subset \bigoplus_{\mu<\omega}\left( \Ind_{\mf{l}_\eta}^\g Y(\chi, \eta)\right)_\mu.
\] 
Therefore, $\left( \bar{\n}^\eta U(\g)\overline{w}  \right)_\omega = 0$, and $ U(\mf{l}_\eta)\overline{w}\cap (\overline{\n}^\eta U(\g))\overline{w}=0$. This establishes equation (\ref{no overlap}), which we now use to show that $\phi\circ\pi_\eta(\Ann_{U(\g)}\overline{w})=0$.

Let $a \in \Ann_{U(\g)}\overline{w}$ and use (\ref{Ul decomp}) to write $a = \pi_\eta(a) + n + m$ for $n \in \overline{\mf{n}}^\eta U(\g)$, $m \in U(\g)\mf{n}^\eta$. Then 
\[
a\overline{w}=(\pi_\eta(a) + n + m) \overline{w} = \pi_\eta(a)\overline{w} + n \overline{w} =0, 
\]
because $m \in \Ann_{U(\g)} \overline{w}$. Hence by (\ref{no overlap})
\[
\pi_\eta(a) \overline{w} = -n \overline{w}=0,
\]
and $\pi_\eta(a) \in \Ann_{U(\mf{l}_\eta)}\overline{w} = \Ann_{U(\mf{l}_\eta)}w$. It follows that $\phi\circ\pi_\eta(\Ann_{U(\g)}\overline{w})=0$, and therefore $\mathrm{ext}_{\pi_\eta}$ is well-defined. %(The last equality follows from \cite[Prop. 2.4(c)]{McDowell}.)
It is then clear that $\mathrm{ext}_{\pi_\eta}$ and $\res_{U(\lev_\eta)}$ are inverse maps, which concludes the proof.

\end{proof}

Theorem \ref{second result} now follows immediately from Proposition \ref{reduction to nondegenerate}.
\end{proof}

\begin{corollary}
\label{verma form}
Let $\lambda \in \mf{h}^*$ and let $M(\lambda) = U(\g) \otimes_{U(\mf{b})}\C_\lambda$ be the corresponding Verma module. There exists a unique contravariant form on $M(\lambda)$ up to scaling. 
\end{corollary}
\begin{proof}
Let $\eta:\mf{n}\rightarrow \C$ be the Lie algebra morphism sending $x \mapsto 0$ for all $x \in \mf{n}$. Then $\mf{l}_\eta=\mf{h}$, and $\mc{N}(\mf{h})_\eta$ is the category of finite-dimensional $U(\h)$-modules. Applying Theorem \ref{second result} to the one-dimensional $U(\h)$-module $\C_\lambda$ proves the corollary. 
\end{proof}
\begin{remark}
Shapovalov defined a contravariant form on a Verma module $M(\lambda)$ by the recipe
\[
\langle u\cdot 1 \otimes 1, v \cdot 1 \otimes 1 \rangle := \lambda (p_0(\tau(v) u)),
\]
where $u,v \in U(\g)$, $1 \otimes 1 \in M(\lambda)$ is the generating highest weight vector, and $p_0:U(\g)=U(\h)\oplus(\overline{\n}U(\g) + U(\g) \n)\rightarrow U(\h)$ is the Harish-Chandra projection (Definition \ref{twisted HC proj}, Remark \ref{HC hom}). Shapovalov showed that this is the unique contravariant form on $M(\lambda)$ with the property that $\langle 1 \otimes 1, 1 \otimes 1 \rangle=1$ \cite{Shapovalov}. Corollary \ref{verma form} implies the existence of the Shapovalov form. 
\end{remark}
In \cite{McDowell}, McDowell defines a class of $U(\g)$-modules which include both the Verma modules $M(\lambda)$ and the Whittaker modules $Y(\chi, \eta)$ from Section \ref{classification of contravariant forms}. These modules are constructed by applying the parabolic induction functor of Definition \ref{induction functor} to irreducible modules in $\mc{N}(\mf{l}_\eta)$ as $\eta$ varies. Specifically, for a Lie algebra morphism $\eta:\mf{n}\rightarrow \C$ and an infinitesimal character $\chi:Z(\mf{l}_\eta)\rightarrow \C$, McDowell defines a $U(\g)$-module 
\begin{equation}
    \label{McDowell}
M(\chi, \eta):=\Ind_{\mf{l}_\eta}^\g Y(\chi, \eta).
\end{equation}
\begin{remark}
There is an alternate definition of $M(\chi, \eta)$ in terms of generators and relations. The $U(\g)$-module $M(\chi, \eta)$ is the $U(\g)$-module generated by $w$ subject to the relations
\begin{enumerate}
    \item $x \cdot w = \eta(x)w$ for all $x \in \n$, and 
    \item $z \cdot w = \chi(z)w$ for all $z \in Z(\mf{l}_\eta)$.
\end{enumerate}
\end{remark}
An immediate consequence of Theorem \ref{second reduction} is the following. 
\begin{corollary}
The vector space of contravariant forms on $M(\chi, \eta)$ has dimension $|W_\eta|$. 
\end{corollary}
This corollary generalizes Shapovalov's results to McDowell's modules $M(\chi, \eta)$.

\section{Example}
\label{example}
In this section, we illustrate the arguments of Section \ref{classification of contravariant forms} more concretely by showing that the vector space of contravariant forms on a nondegenerate $\mf{sl}(2,\C)$-Whittaker module $Y(\chi, \eta)$ is isomorphic to the vector space $\C^2$. For the remainder of this section, set $\mf{g}=\mf{sl}(2,\C)$. Let 
\[
y=\bp 0 & 0 \\ 1 & 0 \ep,\hspace{2mm}  h=\bp 1 & 0 \\ 0 & -1 \ep, \hspace{2mm} x = \bp 0 & 1 \\ 0 & 0 \ep 
\]
be the standard basis of $\mf{g}$, and $\Omega = \frac{1}{2}h^2+h+2yx \in Z(\mf{g})$ the Casimir element, which generates $Z(\mf{g})$. The antiautomorphism $\tau:U(\mf{g})\rightarrow U(\mf{g})$ maps $\tau(x)=y, \tau(h)=h,$ and $\tau(y)=x$.

Let $\eta \in \ch{\mf{n}}$ be the Lie algebra morphism sending $x \mapsto 1$. Since $\Pi=\Sigma^+=\{\alpha\}$ consists of a single root and $\g_\alpha = \C x$, this choice of $\eta$ is nondegenerate. Let  $\chi:Z(\g)\rightarrow \C$ be the algebra homomorphism sending $\Omega \mapsto 0$. Let 
\[
Y(\chi, \eta)=U(\g)\otimes_{Z(\g)\otimes U(\n)} \C_{\chi, \eta},
\]
as in Section \ref{preliminaries and notation}. $Y(\chi,\eta)$ is an irreducible $\g$-module generated by the Whittaker vector $w=1 \otimes 1$.  

Let $\Phi_{Y(\chi, \eta)}$ be the vector space of $\tau$-invariant linear functionals on $U(\g)$ which vanish on $\Ann_{U(\g)}w$. By Proposition \ref{forms-functionals}, $\Phi_{Y(\chi, \eta)}$ is isomorphic to the vector space of contravariant forms on $Y(\chi, \eta)$. Define a map $\psi:\C^2 \rightarrow \Phi_{Y(\chi, \eta)}$ sending $(c_0, c_1) \mapsto \varphi_{c_0, c_1}=:\varphi$ as follows.
\begin{itemize}
    \item On the Poincar\'{e}--Birkhoff--Witt basis element $y^rh^sx^t\in U(\g)$, $r,s,t \in \Z_{\geq 0}$
    \[
    \varphi(y^rh^sx^t)=\eta(x^{r+t})\varphi(h^s)=\varphi(h^s).
    \]
     \item Define $\varphi(h^s)$ inductively: Set $\varphi(1)=c_0$, $\varphi(h)=c_1$, and for $s \geq 2$, define 
    \[
    \varphi(h^s)=-\varphi(2h^{s-1}+4h^{s-2}yx).
    \]
    This is well-defined because rewriting $h^{s-2}yx$ in terms of the Poincar\'{e}--Birkhoff--Witt basis results in a sum whose terms only include powers of $h$ which are equal to or lower than $s-2$. (For example, $h^2yx=yh^2x-4yhx+4yx$.) Hence inductively, $\varphi$ is already defined on $2h^{s-1}+4h^{s-2}yx$. 
    \item Extend linearly to define $\varphi$ on all of $U(\mf{g})$. 
\end{itemize}
It is clear from this construction that if $(c_0, c_1) \neq (d_0, d_1) \in \C^2$, then $\varphi_{c_0, c_1} \neq \varphi_{d_0, d_1}$, so $\varphi$ defines an injection 
\[
\psi: \C^2 \hookrightarrow U(\mf{g})^*. 
\]
\begin{proposition} $\psi:\C^2 \rightarrow \Phi_{Y(\chi, \eta)}$ is an isomorphism of vector spaces. 
\end{proposition}
\begin{proof}
By construction, $\psi$ is linear and injective, so it remains to show that im$\psi=\Phi_{Y(\chi, \eta)}$. We need to show that $\varphi:=\varphi_{c_0, c_1}$ is contained in $\Phi_{Y(\chi, \eta)}$ by showing that it satisfies: (1) $\varphi(\tau(u))=\varphi(u)$, and (2) $\varphi( \Ann_{U(\g)}(w)) =0$. We begin by noting two consequences of the definition of $\varphi$.
\begin{itemize}
    \item For any $u \in U(\mf{g})$, $\varphi(ux)=\eta(x)\varphi(u)=\varphi(u)$. 
    \item For any $u \in U(\mf{g})$, $\varphi(yu)=\eta(x)\varphi(u)=\varphi(u)$.
\end{itemize}
%These observations follow from the fact that $\eta(e^{k+1})=\eta(e)\eta(e^k)$. 
First we check that $\varphi$ satisfies (1) on the Poincar\'{e}--Birkhoff--Witt basis element $y^rh^sx^t$:
\[
\varphi(y^rh^sx^t)=\varphi(h^s)=\varphi(y^th^sx^r)=\varphi(\tau(y^rh^sx^t)). 
\]
Extending linearly, we conclude that $\varphi$ satisfies (1) for any $u \in U(\g)$.

Next we check (2). Note that $\Ann(w)=U(\mf{g})\Omega + U(\mf{g})(x-\eta(x))$ (Proposition \ref{annihilator}), so a generic element of the annihilator is a sum of elements of the form $y^rh^sx^t\Omega + y^ih^jx^k(x-\eta(x))$. Then 
\begin{align*}
\varphi(y^rh^sx^t\Omega + y^ih^jx^k(x-\eta(x)))&=\left(\eta(x^r)\varphi(h^s\Omega x^t)\right) + \left( \varphi(y^ih^jx^{k}x)-\eta(x)\varphi(y^ih^jx^k)\right) \\
&=\eta(x^{r+t})\varphi\left(h^s\left(\frac{1}{2}h^2+h+2yx\right)\right) \\
&= \frac{1}{2}\varphi(h^{s+2})+\varphi(h^{s+1}+2h^syx)\\
&=-\frac{1}{2}\varphi(2h^{s+1}+4h^syx) + \varphi(h^{s+1}+2h^syx)\\
&=0. 
\end{align*}
By extending linearly, we see that for any $a \in \Ann(w)$, $\varphi(a)=0$. This proves that $\mathrm{im}\psi \subseteq \Phi_{Y(\chi, \eta)}$. 
To show that $\text{im}\psi=\Phi_{Y(\chi, \eta)}$, we make the following observations about any $\varphi \in \Phi_{Y(\chi, \eta)}$. 
\begin{itemize}
    \item The value of $\varphi(y^rh^sx^t)$ is completely determined by $\varphi(h^s)$. This follows from three facts: first, any Poincar\'{e}--Birkhoff--Witt basis element $y^rh^sx^t$ can be expressed as $u(x-\eta(x))+ay^rh^s$ for some $u \in U(\mf{g})$ and $a \in \C$ by ``peeling off'' $x$'s (i.e. rewriting $y^rh^sx^t=y^rh^sx^{t-1}(x-\eta(x))+\eta(x)y^rh^sx^{t-1}$); second, $\varphi$ vanishes on the annihilator of $w$; and third, $\varphi$ is $\tau$-invariant.
    \item The value of $\varphi(h^s)$ is completely determined by $\varphi(1)$ and $\varphi(h)$. This follows from the fact that we can rewrite $h^s=2h^{s-2}\Omega-2h^{s-1}-4h^{s-2}yx$, and $\Omega \in \Ann(w)$. 
\end{itemize}
Therefore, a choice of $\varphi(1)$ and $\varphi(h)$ in $\C$ completely determines a linear functional $\varphi \in \Phi_{Y(\chi, \eta)}$, so $\psi:\C^2\rightarrow \Phi_{Y(\chi, \eta)}$ is surjective. 
\end{proof}

Hence the vector space of contravariant forms on the $\mf{sl}_2(\C)$-module $Y(\chi, \eta)$ is 2-dimensional. Note that this is also the cardinality of $W=\Z/2\Z$.

\bibliographystyle{alpha}
\bibliography{contravariant-forms}

\end{document}